\documentclass[reqno]{amsart}
\usepackage{amsthm,amsmath,amssymb}
\usepackage{enumerate}
\usepackage{comment}
\usepackage[colorlinks=true,hypertexnames=false,linkcolor=blue,citecolor=blue]{hyperref}

\theoremstyle{plain}
\newtheorem{theorem}{Theorem}%[section]
\newtheorem{proposition}{Proposition}%[section]
\newtheorem{lemma}{Lemma}%[section]
\newtheorem{claim}{Claim}%[section]
\theoremstyle{definition}
\newtheorem{definition}{Definition}%[section]
\theoremstyle{remark}
\newtheorem{remark}{Remark}%[section]

\includecomment{proof-comment}
\includecomment{pre-comment}

\excludecomment{proof-claim-comment}
\newcommand{\interior}[1]{\ensuremath{\mathring{#1}}}

\title[]{Generalized Whitney topologies are Baire}
\thanks{
This work has been supported by 
``Projeto Tem\'atico Din\^amica e Geometria em Baixas Dimens\~oes'' FAPESP Grant 2016/25053-8, 
FAPESP Grant 2015/17909-7, %PH+EdF
CAPES Grant CSF-PVE-S - 88887.117899/2016-00, %PH
a CAPES/PNPD Grant %PH
and the EU Marie-Curie IRSES Brazilian-European partnership in Dynamical Systems (FP7-PEOPLE-2012-IRSES 318999 BREUDS)}

\author{Edson de Faria} 
\address{Edson de Faria, Instituto de Matem\'{a}tica e Estat\'{i}stica, USP, S\~{a}o Paulo, SP, Brazil}
\email[]{edson@ime.usp.br}

\author{Peter Hazard} 
\address{Peter Hazard, Instituto de Matem\'{a}tica e Estat{\'i}stica, USP, S\~{a}o Paulo, SP, Brazil}
\email[]{pete@ime.usp.br}

\date{\today}

\begin{document}

\newpage
%\pagestyle{plain}
%\setcounter{page}{1}

%
%
%
%%%%%%%%%%%%%%%%%%%%%

%%%%%%%%%%%%%%%%%%%%
%\end{pre-comment}
%%%%%%%%%%%%%%%%%%%%

\begin{abstract}
In this paper we show that certain generalizations of the 
$C^r$-Whitney topology, which include the H\"older-Whitney 
and Sobolev-Whitney topologies on smooth manifolds, 
satisfy the Baire property, to wit, 
the countable intersection of open and dense sets is dense.
\end{abstract}

\maketitle

%%%%%%%%%%%%%%%%%%%%%%%%%%%
\section{Introduction.}
%%%%%%%%%%%%%%%%%%%%%%%%%%%
%

The study of {\it generic\/} properties in topological spaces plays an important role not only in Topology, but also in Analysis, 
Geometry and Dynamics. A property is said to be {\it generic\/} 
in a topological space $X$ if it holds in a {\it residual\/}, or {\it Baire second-category\/}, subset  
of $X$ -- in other words, in a set containing the intersection of a countable family of 
open and dense subsets of $X$. For this notion to be useful, the topology with which $X$ is endowed should 
make it into a {\it Baire space\/} -- that is to say, a space in which all Baire second-category sets are 
dense (see \S \ref{subsect:baireprop} below). 
Thus, in such spaces residual sets are large, whereas their complements -- 
the so-called {\it meager\/} sets -- 
are small. The topological dichotomy residual sets/meager sets is akin to the measure-theoretic dichotomy full-measure sets/null sets, 
and a comparison 
between these dichotomies is often quite fruitful (see~\cite{Oxtoby}).

%When studying a large class $\mathcal{D}$ of dynamical systems, it is often important to understand 
%which phenomena are in a suitable sense {\it typical\/} or {\it generic\/} in this class. 
%This usually requires defining a topology in $\mathcal{D}$. 

In the case when $X$ is the class of $C^r$ smooth maps between two smooth manifolds $M$ and $N$, the most useful topology 
is the $C^r$ compact-open topology, also called the (strong) {\it $C^r$ Whitney topology\/}. It is a classical result that 
the $C^r$ Whitney topology on $C^r(M,N)$ makes it into a Baire space (see 
\cite[Ch.~2]{HirschBook}).
This topology was introduced by H.~Whitney in his study of embeddings and immersions 
of smooth manifolds into Euclidean spaces, and also in his study of stability properties of singularities of mappings. 
Since then, its use in Differential Topology (including Morse Theory), Geometry, and Dynamics, has become pervasive. 

In Dynamics, the generic viewpoint has its origins in the study of structurally stable flows 
on surfaces by M.~Peixoto \cite{Peixoto}, and in the subsequent works by I.~Kupka \cite{Kupka} and S.~Smale \cite{Smale} 
establishing the structural stability 
of what are now known as Kupka-Smale diffeomorphisms. 

In the present paper, we prove the Baire property for certain generalizations of Whitney topologies that are tailor-made for spaces of mappings 
of low regularity (such as H\"older or Sobolev mappings). 
The classical proof of the Baire property for the strong $C^r$ Whitney topology uses jets,
but it is not clear what replaces this notion for more general classes of mappings.
The approach we adopt below circumvents this problem.

Generalized Whitney topologies provide the correct framework for the study of generic properties of low-regularity dynamical systems. 
We have in mind specific applications to the study of {\it topological entropy\/} for these low-regularity systems. 
Explicit examples are provided by our papers \cite{deFariaHazardTresser2017a,deFariaHazardTresser2017b} in collaboration with C.~Tresser. 
See \S \ref{subsect:main} for our main result (Theorem \ref{thm:Whitney_topologies_are_Baire}) and \S \ref{sect:examples} for 
the relevant examples.

\begin{remark}\label{rmk:mappings_not_homeos}
We will only consider spaces of (possibly non-invertible) mappings, rather than spaces of homeomorphisms.
However, with obvious changes the same results go through 
in the latter case as well.
\end{remark}
%

%%%%%%%%%%%%%%%%%%%%%%%%%%%
\section{Preliminaries.}
%%%%%%%%%%%%%%%%%%%%%%%%%%%

%%%%%%%%%%%%%%%%%%%%%%%%%%%%%%%%%
\subsection{The Baire property.}\label{subsect:baireprop}
%%%%%%%%%%%%%%%%%%%%%%%%%%%%%%%%%

Recall that a topological space $X$ {\it satisfies the Baire property\/}, or {\it is a Baire space\/}, if the 
intersection of any countable family of open and dense subsets of $X$ is dense in $X$. 
The Baire category theorem states that if $X$ is either a complete metric space or a locally compact Hausdorff space, then 
$X$ is a Baire space. (See, e.g.,~\cite[Theorem 34]{KelleyBook}.) 
In the present paper we shall make use of a straightforward generalization of this theorem, 
which we proceed to present.

We start with a definition. 
Let $(X,\tau)$ be a topological space, and 
let $\rho$ be a metric on the set $X$, 
which may or may not be compatible with the topology $\tau$. 
Given $E\subset X$ we will denote 
by $\tau|_E$ the induced topology on $E$,
by $\mathrm{int}_\rho(E)$ the interior of $E$ in the topology induced by $\rho$,
and by $\mathrm{diam}_\rho(E)$ the diameter of $E$ with respect to $\rho$.

\begin{definition} 
A $\tau$-open set $V\subset X$ is 
\emph{$\rho$-sprinkled} 
if for 
each $\tau$-open $U\subset V$ and
each $r>0$, 
there exists 
a $\rho$-closed set $E\subset U$ with 
$\rho$-diameter $\mathrm{diam}_\rho(E)<r$, 
whose 
$\tau$-interior $\mathrm{int}_\tau(E)$ is non-empty.
\end{definition}
\begin{lemma}\label{lem:sprinkled}
Let 
$(X,\tau)$ be a topological space and 
let 
$\rho$ be a complete metric on $X$ 
such that
every $\tau$-open set in  
$(X,\tau)$ is $\rho$-sprinkled. 
Then $(X,\tau)$ is a Baire space. 
\end{lemma}
\begin{remark}
The property of being $\rho$-sprinkled is hereditary. 
Thus, from the proof given below, 
to show the Baire property it is enough to 
have an open covering by $\rho$-sprinkled sets.
\end{remark}
\begin{proof}
Let 
$\{G_n\}_{n\in\mathbb{N}}$ 
be a sequence of $\tau$-open-and-dense subsets of $X$, 
and let 
$V\subset X$ 
be $\tau$-open and non-empty. 
Since 
$V\cap G_1$ 
is non-empty and $\tau$-open, 
by the sprinkling property there exists 
$E_1\subset V\cap G_1$ 
which is $\rho$-closed, with 
$\mathrm{diam}_\rho(E_1)<\frac{1}{2}$ 
and 
$\mathrm{int}_\tau(E_1)$ non-empty. 
By induction, suppose that 
$E_n\subset \mathrm{int}_\tau(E_{n-1})\cap G_n$ 
has been defined so that it is $\rho$-closed, 
has 
$\rho$-diameter $\mathrm{diam}_\rho(E_n)<\frac{1}{2^n}$ 
and 
$\mathrm{int}_\tau(E_{n})$ is non-empty. 
By the sprinkling property, since 
$\mathrm{int}_\tau(E_{n})\cap G_{n+1}$ is $\tau$-open and non-empty, 
there exists 
$E_{n+1}\subset \mathrm{int}_\rho(E_{n})\cap G_{n+1}$ 
which is $\rho$-closed, 
has $\rho$ diameter $\mathrm{diam}_\rho(E_{n+1})<\frac{1}{2^{n+1}}$, 
and whose $\tau$-interior is non-empty. 
By construction, we have 
$V\supset E_1\supset E_2\supset \cdots \supset E_n\supset \cdots$. 
Selecting 
$x_n\in E_n$ for each $n$, we see that $(x_n)_{n\in\mathbb{N}}$ is a Cauchy sequence for the complete metric $\rho$. 
Therefore the $\rho$-limit $x=\lim_{n\to\infty}x_n$ exists. 
Since each $E_n$ is $\rho$-closed, and 
since $x_k\in E_n$ for all $k\geq n$, we deduce that 
\begin{equation}
x\;\in\; \bigcap_{n\in\mathbb{N}} E_n\;\subset V\cap \bigcap_{n\in\mathbb{N}} G_n\ .
\end{equation}
Thus $\bigcap_{n\in\mathbb{N}} G_n$ is $\tau$-dense in $X$, and so $(X,\tau)$ is a Baire space as claimed. 
\end{proof}
\begin{definition}
The topological space $(X,\tau)$ is 
{\it locally sprinkled} 
if 
for each non-empty $\tau$-open set $Y$ 
there exists 
a set $Z$ contained in $Y$, 
with non-empty $\tau$-interior $\mathrm{int}_\tau(Z)$,
and a complete metric $\rho_Z$ on $Z$, 
such that 
$\mathrm{int}_\tau(Z)$ is $\rho_Z$-sprinkled.
\end{definition}
\begin{lemma}\label{lem:locally_sprinkled}
A locally sprinkled topological space $(X,\tau)$ is Baire.
\end{lemma}
\begin{proof}
Let 
$\{G_n\}_{n\in\mathbb{N}}$ 
be a countable collection of $\tau$-open-and-dense subsets of $X$.
To show that 
$\bigcap_{n\in\mathbb{N}}G_n$ 
is $\tau$-dense, 
it suffices to show that 
$\bigcap_{n\in\mathbb{N}}G_n$ 
intersects each $\tau$-open set $Y$ in $X$.
By hypothesis, $Y$ contains a set $Z$ 
with non-empty $\tau$-interior $\mathrm{int}_\tau(Z)$ 
which is $\rho_Z$-sprinkled by some complete metric $\rho_Z$ on $Z$.
By the same argument as in the preceding Lemma, 
$Z\cap\bigcap_{n\in\mathbb{N}}G_n$ is non-empty and the result follows.
\end{proof}

%%%%%%%%%%%%%%%%%%%%%%%%%%%%%%
\subsection{Families of norms}\label{subsect:family_of_norms}
%%%%%%%%%%%%%%%%%%%%%%%%%%%%%%
Assume that to each compact  
$K\subset\mathbb{R}^m$ 
we assign a semi-norm $[\,\cdot\,]_K$ on 
$C^0(K,\mathbb{R}^d)$. 
Denote the family of such semi-norms by $\mathcal{F}$.
For each compact subset $K$ of $\mathbb{R}^m$,
define the norm
\begin{equation}
\|\cdot\|_K
=
\|\cdot\|_{C^0(K,\mathbb{R}^d)}+[\,\cdot\,]_K \ .
\end{equation}
We will refer to this as the $\mathcal{F}$-norm 
corresponding to the compact set $K$.
Let 
$C^\mathcal{F}(K,\mathbb{R}^d)\subset C^0(K,\mathbb{R}^d)$ 
denote the subspace of mappings for which 
$\|\cdot\|_K$ 
is finite.
Observe that 
$C^\mathcal{F}(K,\mathbb{R}^d)$, 
endowed with this norm, is a Banach space.
In what follows we will assume that, 
for each compact set $K$, the semi-norm 
$[\,\cdot\,]_K$ 
satisfies the following properties:
\begin{enumerate}
\item[(I)] {\sl Composition property\/}.
For any smooth diffeomorphisms onto their images $\psi$ and $\varphi$,
there exists a positive real number $\kappa$
such that
\begin{equation*}
\left[\psi\circ g\circ\varphi^{-1}-\psi\circ g'\circ\varphi^{-1}\right]_{\varphi(K)}
\leq
\kappa \|g-g'\|_K \ ,
\quad
\forall g,g'\in C^\mathcal{F}\left(K,\mathbb{R}^d\right) \ .
\end{equation*}
Moreover, 
for $\|g-g'\|_K$ sufficiently small,
$\kappa$ depends only upon  
the derivatives of $\psi$ and $\varphi$ and on $\max \left\{[g]_K,[g']_K\right\}$.
\item[(II)] {\sl Gluing property\/}.
Given a finite collection of compact sets 
$\{K_s\}_{s\in S}$, 
let 
$K=\bigcup_{s\in S}K_s$.
There exists a positive real number $C$, 
depending upon 
$\{K_s\}_{s\in S}$ 
only, such that
\begin{equation*}
[g]_K\leq C \sum_{s\in S} [g]_{K_s} \ , 
\quad 
\forall g\in C^\mathcal{F}\left(K,\mathbb{R}^d\right) \ .
\end{equation*}
\item[(III)] {\sl Monotonicity property\/}.
If $K'\subset K$ then 
\begin{equation*}
[g]_{K'}
\leq 
[g]_{K} \ , 
\quad 
\forall g\in C^\mathcal{F}\left(K,\mathbb{R}^d\right) .
\end{equation*}
\end{enumerate}
\begin{remark}
In property (I) we require $\psi$ and $\varphi$ to be of class 
$C^\infty$ as below we will consider when they are the coordinates or transition maps
of manifolds of class $C^\infty$. 
If we consider manifolds of class $C^r$, or manifolds 
whose transition maps lie in some other pseudogroup, 
then property (I) must be changed accordingly. 
However, we only consider the smooth case as this 
simplifies the statements in Section~\ref{sect:examples}.
\end{remark}
\begin{remark}
The sup-norm $\|\cdot\|_{C^0(K,\mathbb{R}^d)}$ satisfies properties (I)--(III). 
Consequently, the norm $\|\cdot\|_K$ also satisfies (I)--(III). 
\end{remark}
\begin{remark}\label{rmk:sums_of_semi-norm_families}
Any positive linear combination of semi-norms satisfying (I)--(III) will also satisfy (I)--(III).
Thus the families of semi-norms $\mathcal{F}$ satisfying (I)--(III) form a positive cone in the space all such families. 
\end{remark}
%
%
%%%%%%%%%%%%%%%%%%%%%%%%%%%%%%%%%%%
\subsection{Whitney topologies}
%%%%%%%%%%%%%%%%%%%%%%%%%%%%%%%%%%%
Let $M$ and $N$ be smooth manifolds.
Let 
$C^\mathcal{F}(M,N)$ 
denote the set of mappings 
$f\in C^0(M,N)$ 
so that in local charts, 
$f$ has finite $\mathcal{F}$-norm on each compact subset. 
Let
$(U,\varphi)$ be a chart in $M$;
$(V,\psi)$ be a chart in $N$; 
$K$ a compact subset of $U$; 
$\epsilon$ a positive extended real number; 
and let
$f\in C^\mathcal{F}(M,N)$ satisfy 
$f(K)\subset V$.
Denote by
\begin{equation}\label{def:subbasic}
\mathcal{N}\left(f;(U,\varphi),(V,\psi),K,\epsilon\right)
\end{equation}
the set of maps 
$g\in C^\mathcal{F}(M,N)$ 
satisfying 
$g(K)\subset V$ 
and
\begin{equation}\label{ineq:N-ball1}
\left\|\psi \circ f\circ\varphi^{-1}-\psi\circ g\circ\varphi^{-1}\right\|_{\varphi(K)}
<
\epsilon \ .
\end{equation}
%
%%%%%%%%WEAK WHITNEY%%%%%%%%
We call the topology generated by 
the sets~\eqref{def:subbasic} 
%({\it i.e.}, for which these sets are sub-basic)
the {\it weak $\mathcal{F}$-Whitney topology}.
We call sets of the form~\eqref{def:subbasic}, 
{\it weak sub-basic neighbourhoods}.
%
%
%
%%%%%%%%STRONG WHITNEY%%%%%%%%
Let 
\begin{itemize}
\item 
$\Phi=\{(U_t,\varphi_t)\}_{t\in T}$ be a {\it locally finite} collection of charts on $M$,
%({\it i.e.\/}, any point $x\in M$ is contained in at most finitely many $U_i$)
\item 
$\Psi=\{(V_t,\psi_t)\}_{t\in T}$ a collection of charts on $N$,
\item 
$K=\{K_t\}_{t\in T}$ be a collection of compact sets, $K_t\subset U_t$, 
\item 
$\epsilon=\{\epsilon_t\}_{t\in T}$ be a collection of positive extended real numbers,
\end{itemize}
and let $f\in C^\mathcal{F}(M,N)$ satisfy $f(K_t)\subset V_t$ for each $t\in T$.
Define the {\it strong basic neighbourhood}
\begin{equation}\label{def:strong_basic_nhd}
\mathcal{N}_{s}(f;\Phi,\Psi,K,\epsilon)
\end{equation}
to be the collection of maps $g\in C^\mathcal{F}(M,N)$ such that, 
for each $t\in T$, 
$g(K_t)\subset V_t$
and
\begin{equation}
\left\|\psi_t\circ f\circ\varphi_t^{-1}-\psi_t\circ g\circ\varphi_t^{-1}\right\|_{\varphi_t(K_t)}
<
\epsilon_t \ .
\end{equation}  
The {\it strong $\mathcal{F}$-Whitney topology} is the topology 
with the collection of all strong basic neighbourhoods as a base.
\subsection{Basic properties.}
%%%%%%%%%%%%%%%%%%%%%%%%%%%%%%%%%
We begin with the following straightforward observation.
\begin{proposition}
\begin{enumerate}
\item
The weak and strong $\mathcal{F}$-Whitney topologies are Hausdorff.
\item
The strong $\mathcal{F}$-Whitney topology is finer than 
the weak $\mathcal{F}$-Whitney topology, 
and they coincide when $M$ is compact. 
\end{enumerate}
\end{proposition}
\begin{remark}\label{rmk:locally_finite_covers_are_countable}
Given a collection of charts $\Phi$ of $M$ which are 
locally finite, 
since the manifolds we consider are second countable 
(and hence each subspace of $M$ is also second countable) 
%(and thus are Lindel\"of)
it follows that this collection $\Phi$ is countable.
%%%%%%%%%%%%%%%
\begin{comment}
%%%%%%%%%%%%%%%
Recall a space is {\it Lindel\"of} if every open 
cover possesses a countable open subcover.  
No uncountable open cover in a Lindel\"of space is 
locally finite.

{\it Proof\/}:
Let $\mathcal{U}$ be an open cover of $X$.
For each $x\in X$, let $V_x$ be an open neighbourhood of $x$ in $X$ such that $V_x$ intersects finitely many $U\in\mathcal{U}$.
Then $\mathcal{V}=\{V_x\}_{x\in X}$ is an open cover of $X$.
The Lindel\"of property implies that there exists a countable subcover, by $V_{x_1}, V_{x_2},\ldots$ say.
But $V_{x_k}$ intersects finitely many $U\in\mathcal{U}$.
Thus $X$ intersects the collection $\mathcal{U}$ at only countably many $U\in\mathcal{U}$.
/\!/
%%%%%%%%%%%%%
\end{comment}
%%%%%%%%%%%%%
\end{remark}
\begin{remark}\label{rmk:cpt_cover_with_covering_interiors}
Given any cover of $M$ by compact sets 
$\{K_t\}_{t\in T}$, 
where each $K_t$ lies in some open set $U_t$ 
(in the cases we consider this will be the domain of some chart), 
we can construct a new covering 
$\{K_t'\}_{t\in T}$ 
of $M$ by compact sets, 
so that $K_t'$ also lies in $U_t$, 
and with the additional property that the interiors 
$\{\interior{K}_t'\}_{t\in T}$ 
form an open cover of $M$.
\end{remark}
\begin{remark}
Given charts $(U,\varphi)$ and $(V,\psi)$ of $M$ and $N$ respectively,
a mapping $f\in C^\mathcal{F}(M,N)$,
and compact sets $K\subset K'$ in $U$ so that $f(K')\subset V$,
for any $\epsilon>0$ we have
\begin{equation}
\mathcal{N}(f;(U,\varphi),(V,\psi),K',\epsilon)\subset
\mathcal{N}(f;(U,\varphi),(V,\psi),K,\epsilon) \ .
\end{equation}
\end{remark}
\begin{proposition}\label{prop:change_of_charts-local}
Fix
a chart $(U,\varphi)$ in $M$, 
a chart $(V,\psi)$ in $N$, and 
a set $K$ compact in $U$.
Take
\begin{itemize}
\item
$\{(U_t,\varphi_t)\}_{t\in T}$ a collection of charts in $M$, 
\item
$\{(V_t,\psi_t)\}_{t\in T}$ a collection of charts in $N$, 
\item
$\{K_t\}_{t\in T}$ a collection of compact sets covering $K$, with $K_t$ compact in $U_t$.
\end{itemize}
There exists a positive real numbers 
$C_0$, $C_\sigma$, and $C$, 
depending upon 
$K$, 
the collection of $K_t$ intersecting $K$ and 
the collections of corresponding charts only,
such that for all suitably defined $g$ and $g'$,
\begin{align}
&
\|\psi\circ g\circ\varphi^{-1}
-\psi\circ g'\circ\varphi^{-1}\|_{C^0(\varphi(K),\mathbb{R}^d)}\notag\\
&\leq
C_0
\sup_{t\in T} 
\|\psi_t\circ g\circ\varphi_t^{-1}
-\psi_t\circ g'\circ\varphi_t^{-1}\|_{C^0(\varphi_t(K_t),\mathbb{R}^d)}\tag{a}\\
&
[\psi\circ g\circ\varphi^{-1}
-\psi\circ g'\circ\varphi^{-1}]_{\varphi(K)}\notag\\
&\leq
C_\sigma\sum_{t\in T : K\cap K_t\neq\emptyset}
\|\psi_t\circ g\circ\varphi_t^{-1}
-\psi_t\circ g'\circ\varphi_t^{-1}\|_{\varphi_t(K_t)}\tag{b}\\
&
\|\psi\circ g\circ\varphi^{-1}
-\psi\circ g'\circ\varphi^{-1}\|_{\varphi(K)}\notag\\
&\leq
C
\sum_{t\in T : K\cap K_t\neq\emptyset} 
\|\psi_t\circ g\circ\varphi_t^{-1}
-\psi_t\circ g'\circ\varphi_t^{-1}\|_{\varphi_t(K_t)}\tag{c}
\end{align}
\end{proposition}
\begin{proof}
%%%%%%%%%%%%%%%
%\begin{comment}
%%%%%%%%%%%%%%%
%ONLY DEPENDS ON the sup-norm satisfying the composition and monotonicity properties.
Inequality (a) follows by a straightforward application of 
the Lipschitz property of the smooth diffeomorphism
$\psi\circ\psi_t^{-1}$ 
over the compact set 
$K\cap K_t$, 
and then taking the supremum over all $t\in T$.
Setting 
$C_0=\sup_{t\in T}[\psi\circ\psi_t^{-1}]_\mathrm{Lip}$ 
gives the required inequality.
%%%%%%%%%%%%%%%
\begin{comment}
%%%%%%%%%%%%%%%
For 
$x\in\varphi(K\cap K_t)$,
\begin{align}
&
|\psi\circ g\circ\varphi^{-1}(x)-\psi\circ g'\circ\varphi^{-1}(x)|\notag\\
&\leq
|\psi\circ\psi_t^{-1}\circ\psi_t\circ g\circ\varphi_t^{-1}\circ\varphi_{t}\circ\varphi^{-1}(x)
-\psi\circ\psi_t^{-1}\circ\psi_t\circ g'\circ\varphi_t^{-1}\circ\varphi_{t}\circ\varphi^{-1}(x)|\\
%&\leq
%[\psi\circ\psi_t^{-1}]_{\mathrm{Lip}}
%|\psi_t\circ g\circ\varphi_t^{-1}(y)-\psi_t\circ g'\circ\varphi_t^{-1}(y)|\\
&\leq
[\psi\circ\psi_t^{-1}]_{\mathrm{Lip}}
\|\psi_t\circ g\circ\varphi_t^{-1}
-\psi_t\circ g'\circ\varphi_t^{-1}\|_{C^0(\varphi_t(K\cap K_t),\mathbb{R}^d)}
\end{align}
%%%%%%%%%%%%%
\end{comment}
%%%%%%%%%%%%%
Next, by the Gluing property (II)
\begin{align}
&
[\psi\circ g\circ\varphi^{-1}-\psi\circ g'\circ\varphi^{-1}]_{\varphi(K)}\notag\\
&\leq
C\sum_{t\in T} 
[\psi\circ g\circ\varphi^{-1}-\psi\circ g'\circ\varphi^{-1}]_{\varphi(K\cap K_t)}
\end{align}
where $C$ depends only upon 
$K$ 
and the subcollection of
$\{K_t\}_{t\in T}$, 
consisting of sets intersecting $K$.
By 
the Composition property (I) and 
the Monotonicity property (III),
\begin{align}
&
[\psi\circ g\circ\varphi^{-1}
-\psi\circ g'\circ\varphi^{-1}]_{\varphi(K\cap K_t)}\notag\\
%&=
%[\psi\circ\psi_t^{-1}\circ\psi_t\circ g\circ\varphi_t^{-1}\circ\varphi_t\circ\varphi^{-1}
%-\psi\circ\psi_t^{-1}\circ\psi_t\circ g'\circ\varphi_t^{-1}\circ\varphi_t\circ\varphi^{-1}]_{\varphi(K\cap K_t)}\\
&\leq
\kappa
\|\psi_t\circ g\circ\varphi_t^{-1}
-\psi_t\circ g'\circ\varphi_t^{-1}\|_{\varphi_t(K\cap K_t)}
\\
&\leq
\kappa
\|\psi_t\circ g\circ\varphi_t^{-1}-\psi_t\circ g'\circ\varphi_t^{-1}\|_{\varphi_t(K_t)}
\end{align}
where $\kappa$ is a positive real number 
depending only upon the Lipschitz constants of 
$\psi\circ\psi_t^{-1}$ 
and 
$\varphi_t^{-1}\circ\varphi$.
Thus inequality (b) follows.
Finally (c) follows trivially from (a) and (b).
%%%%%%%%%%%%%
%\end{comment}
%%%%%%%%%%%%%
\end{proof}
As the strong basic sets~\eqref{def:strong_basic_nhd} form a base 
for the strong $\mathcal{F}$-Whitney topology, 
arbitrary open sets are unions of sets of this form. 
In fact, we can say more.
\begin{lemma}\label{lem:change_of_charts}
Let $\mathcal{U}$ be an open set in the strong $\mathcal{F}$-Whitney topology and $f\in \mathcal{U}$.
Let 
\begin{itemize}
\item
$\Phi=\{(U_t,\varphi_t)\}_{t\in T}$ be a locally finite collection of charts on $M$,
\item
$\Psi=\{(V_t,\psi_t)\}_{t\in T}$ be a collection of charts on $N$,
\item
$K=\{K_t\}_{t\in T}$ be a collection of compact subsets of $M$ covering $M$ 
such that 
$K_t\subset U_t$ and
$f(K_t)\subset V_t$, for all $t\in T$.
\end{itemize}
Then there exists a family of positive extended real numbers
$\epsilon=\{\epsilon_t\}_{t\in T}$
such that
\begin{equation}
\mathcal{W}=\mathcal{N}_s(f;\Phi,\Psi,K,\epsilon)
\end{equation}
is contained in $\mathcal{U}$.
\end{lemma}
\begin{proof}
%%%%%%%%%%%%%%%
%\begin{comment}
%%%%%%%%%%%%%%%
Since the sets~\eqref{def:strong_basic_nhd} form a base, 
there exists a basic set of the form
\begin{equation}
\mathcal{V}=\mathcal{N}_s(f;\Phi',\Psi',K',\epsilon')
\end{equation}
which is contained in $\mathcal{U}$.
Thus it suffices to show that $\mathcal{V}$ contains a set 
of the form $\mathcal{W}$ above.
It suffices to show that, 
for suitable $\epsilon=\{\epsilon_t\}_{t\in T}$, 
if $g$ satisfies
\begin{equation}
\|\psi_t\circ f\circ\varphi_t^{-1}-\psi_t\circ g\circ\varphi_t^{-1}\|_{\varphi_t(K_t)}<\epsilon_t
\end{equation}
for each $t\in T$, then
\begin{equation}\label{ineq:Phi'Psi'}
\|\psi'_s\circ f\circ(\varphi_s')^{-1}-\psi'_s\circ g\circ(\varphi_s')^{-1}\|_{\varphi_s'(K_s')}<\epsilon_s' \ .
\end{equation}
Fix $s\in S$.
Observe that, as $\Phi$ is locally finite, 
the set $K_s'$ is covered by at most finitely many $K_t$.
Consequently, 
by Proposition~\ref{prop:change_of_charts-local}(c),
there exists 
positive $\delta_s$ 
such that, 
for all $t\in T$ with $K_t$ intersecting $K$,
if
\begin{equation}
\|\psi_t\circ f\circ\varphi_t^{-1}-\psi_t\circ g\circ\varphi_t^{-1}\|_{\varphi_t(K_t)}<\delta_s
\end{equation}
then inequality~\eqref{ineq:Phi'Psi'} above is satisfied.
For $t\in T$, 
take 
$\epsilon_t=\min_{s\in S: K_s'\cap K_t\neq \emptyset}\delta_s$.
%%%%%%%%%%%%%
%\end{comment}
%%%%%%%%%%%%%
\end{proof}
%
%%%%%%%%%%%%%%%%%%%%%%%%%%%
\section{Main Construction}
%%%%%%%%%%%%%%%%%%%%%%%%%%%
%
%%%%%%%%%%%%%%%%%%%%%%%%%%%%%%%%%%%
\subsection{Closed sub-basic sets.}\label{sect:subbasic_sets}
%%%%%%%%%%%%%%%%%%%%%%%%%%%%%%%%%%%
Throughout the rest of this article
we will use the notation 
$\overline{\mathcal{N}}\left(f;(U,\varphi),(V,\psi),K,\epsilon\right)$
%\begin{equation}\label{def:subbasic_bar}
%\overline{\mathcal{N}}\left(f;(U,\varphi),(V,\psi),K,\epsilon\right)
%\end{equation}
to denote the set of maps 
$g\in C^\mathcal{F}(M,N)$, 
such that $g(K)\subset V$, 
and inequality~\eqref{ineq:N-ball1} above is satisfied with $<$ replaced by $\leq$.
%i.e.,
%\begin{equation}
%\|\psi \circ f\circ\varphi^{-1}-\psi\circ g\circ\varphi^{-1}\|_{\varphi(K)}\leq\epsilon
%\end{equation}
\begin{proposition}\label{prop:barN-closed}
Provided that 
\begin{equation}
\epsilon<\mathrm{dist}\left(\psi\circ f(K),\mathrm{bd}(\psi(V))\right) \ ,
\end{equation}
each of the sets 
$\overline{\mathcal{N}}\left(f;(U,\varphi),(V,\psi),K,\epsilon\right)$
is closed in both the weak and strong $\mathcal{F}$-Whitney topologies, 
with interior equal to
$\mathcal{N}\left(f;(U,\varphi),(V,\psi),K,\epsilon\right)$.
\end{proposition}
\begin{definition}\label{def:special_sub-basic_nhd}
Take $f\in C^\mathcal{F}(M,N)$.
A {\it special sub-basic neighbourhood about $f$} 
is any non-empty set of the form
\begin{equation}\label{eq:V-basic'}
\mathcal{B}
=
\bigcap_{t\in T} 
\overline{\mathcal{N}}(f;(U_t,\varphi_t),(V_t,\psi_t),K_t,\epsilon_t)
\end{equation}
where 
\begin{enumerate}
\item[(i)]
$T$ is a countable index set,
\item[(ii)]\label{property:locally_finite}
$\{(U_t,\varphi_t)\}_{t\in T}$ is a locally finite collection of charts on $M$,
\item[(iii)]
$\{K_t\}_{t\in T}$ is collection of compact sets whose interiors cover $M$,
\item[(iv)]\label{eq:compact_away_from_bd}
$\epsilon_t<\mathrm{dist}(\psi_t\circ f(K_t),\mathrm{bd}\psi_t(V_t))$, for all $t\in T$,
\item[(v)]\label{property:totally_bounded}
$\sup_{t\in T} \epsilon_t<\infty \ $.
\end{enumerate}
\end{definition}
\begin{lemma}\label{lem:opens_contain_specials}
For each non-empty countable intersection of weak sub-basic sets
\begin{equation}
\mathcal{B}_0
=
\bigcap_{t\in T_0} 
\mathcal{N}(f_t;(U_t,\varphi_t),(V_t,\psi_t),K_t,\epsilon_t)
\end{equation}
whose collection of charts $\{(U_t,\varphi_t)\}_{t\in T_0}$ is locally finite, and for 
each $f\in\mathcal{B}_0$,
there exists a special sub-basic neighbourhood 
$\mathcal{B}$ about $f$ which is contained in $\mathcal{B}_0$. 
\end{lemma}
\begin{proof}
%%%%%%%%%%%%%%%%
%\begin{comment}
%%%%%%%%%%%%%%%%
Firstly,
for each $t\in T_0$, 
choose a compact subset $K_t'$ of $M$
with non-empty interior such that 
$K_t\subset K_t'\subset U_t$, 
and
$f_t(K_t')\subset V_t$.
By Remark~\ref{rmk:cpt_cover_with_covering_interiors},
\begin{equation}
\mathcal{N}(f_t;(U_t,\varphi_t),(V_t,\psi_t),K_t',\epsilon_t)
\subset 
\mathcal{N}(f_t;(U_t,\varphi_t),(V_t,\psi_t),K_t,\epsilon_t) \ .
\end{equation}
For each $t\in T_0$,
shrinking $K_t'$ slightly if necessary, we may assume that
$f$ lies in 
$\mathcal{N}(f_t;(U_t,\varphi_t),(V_t,\psi_t),K_t',\epsilon_t)$.
Thus
there exists a positive 
$\epsilon_t'<\epsilon_t$ 
so that
\begin{equation}
\overline{\mathcal{N}}(f;(U_t,\varphi_t),(V_t,\psi_t),K_t',\epsilon_t')
\subset 
\mathcal{N}(f_t;(U_t,\varphi_t),(V_t,\psi_t),K_t,\epsilon_t) \ .
\end{equation}
Shrinking $\epsilon_t'$ if necessary we may assume further that inequality (iv) is satisfied.
Consequently
\begin{equation}
\bigcap_{t\in T_0}\overline{\mathcal{N}}(f;(U_t,\varphi_t),(V_t,\psi_t),K_t',\epsilon_t')
\subset 
\bigcap_{t\in T_0}\mathcal{N}(f_t;(U_t,\varphi_t),(V_t,\psi_t),K_t,\epsilon_t) \ .
\end{equation}
The collection of closed sub-basic sets satisfy (i), (ii) and (iv).
We may add to this collection a countable collection of closed 
sub-basic sets so that it also satisfies (i)--(iv). 
Decreasing each $\epsilon_t$ is necessary,
we may assume that they also satisfy (v).
This gives the result.
%%%%%%%%%%%%%%%
\begin{comment}
%%%%%%%%%%%%%%%
More precisely, we add to the collection
Next, we add to the collection of closed sub-basic sets
\begin{equation}
\overline{\mathcal{N}}(f;(U_t,\varphi_t),(V_t,\psi_t),K_t',\epsilon_t'), \qquad t\in T_0
\end{equation}
a collection
(for some countable index set $T_1$),
\begin{equation}
\overline{\mathcal{N}}(f;(U_t,\varphi_t),(V_t,\psi_t),K_t',\epsilon_t'), \qquad t\in T_1
\end{equation}
where 
$\{(U_t,\varphi_t)\}_{t\in T_1}$ is a collection of charts on $M$,
$\{(V_t,\psi_t)\}_{t\in T_1}$ is a collection of charts on $N$,
$\{K_t'\}_{t\in T_1}$ is a finite collection of compact sets so that 
$\{\interior{K}_t\}_{t\in T_0\cup T_1}$ 
covers $M$,
$K_t'\subset U_t$, $f(K_t')\subset V_t$, 
and 
$\{\epsilon_t'\}_{t\in T_1}$ satisfy inequality~\eqref{eq:compact_away_from_bd} above, for each $t\in T_1$.
Since $f$ lies in each of these sub-basic sets it lies in their intersection, which therefore is non-empty. 
Thus the result follows by setting $T=T_0\cup T_1$.
%%%%%%%%%%%%%
\end{comment}
%%%%%%%%%%%%%
%
%%%%%%%%%%%%%%
%\end{comment}
%%%%%%%%%%%%%%
\end{proof}
\noindent
Let $f\in C^\mathcal{F}(M,N)$ and
take a special sub-basic neighbourhood $\mathcal{B}$ about $f$
as given by Definition~\ref{def:special_sub-basic_nhd} above.
For each $t\in T$,
$C^\mathcal{F}(\varphi_t(K_t),\mathbb{R}^d)$ is a Banach space.
Therefore the sum
\begin{equation}
\mathcal{Z}
=
\left\{
(g_t)_{t\in T}\in \bigoplus_{t\in T} C^\mathcal{F}(\varphi_t(K_t),\mathbb{R}^d) 
\ : \ 
\sup_{t\in T} \|g_t\|_{\varphi_t(K_t)}<\infty
\right\}
\end{equation}
is also Banach space when endowed with the norm
\begin{equation}
\|\cdot\|_{\mathcal{Z}}
=
\sup_{t\in T} \|\cdot\|_{\varphi_t(K_t)} \ .
\end{equation}
Consider the map
\begin{equation}
\iota\colon \mathcal{B}\to \mathcal{Z} \ , 
\qquad 
g\mapsto \bigl(\psi_t\circ f\circ \varphi_t^{-1}-\psi_t\circ g\circ\varphi_t^{-1}\bigr)_{t\in T} \ .
\end{equation}
By property (v) %~\ref{property:totally_bounded}
in Definition~\ref{def:special_sub-basic_nhd},
this map is well-defined.
Denote by $\rho$ the pullback via $\iota$ of the distance induced by the norm $\|\cdot\|_\mathcal{Z}$.
By Property (iii) of Definition~\ref{def:special_sub-basic_nhd},
the sets $\{K_t\}_{t\in T}$ cover $M$, 
and so
$\iota$ is an injection.
Hence $\rho$ defines a metric (rather than just a semi-metric).
\begin{lemma}\label{lem:complete_metric}
If 
$\mathcal{B}$ is a special sub-basic neighbourhood 
then the metric
$\rho$ is complete.
\end{lemma}
\begin{remark}
The metric $\rho$ is, in general, not compatible with the 
either the induced weak or induced strong $\mathcal{F}$-Whitney topologies. 
\end{remark}
\begin{proof}
%%%%%%%%%%%%%%%
%\begin{comment}
%%%%%%%%%%%%%%%
Property (iv), together with Proposition~\ref{prop:barN-closed}, 
implies that, for each $t\in T$, 
$\overline{\mathcal{N}}(f;(U_t,\varphi_t),(V_t,\psi_t),K_t,\epsilon_t)$
is a closed subset.
Thus, 
%since the arbitrary intersection of closed sets is closed, 
$\mathcal{B}$ is closed.
\begin{claim}
The image $\iota(\mathcal{B})$ is closed in $\mathcal{Z}$.
\end{claim}
%
%%%%%%%%%%%%%%%
%\begin{comment}
%%%%%%%%%%%%%%%

\noindent
{\it Proof of Claim\/}:
For an arbitrary sequence in $\iota(\mathcal{B})$, 
convergent in $\mathcal{Z}$, 
we will show the limit is also in $\iota(\mathcal{B})$.
Observe that any sequence in 
$\iota(\mathcal{B})$ 
has the form $\iota(g_k)$ for some sequence $g_k$ in $\mathcal{B}$.
Assume that $\iota(g_k)$ converges to 
$\gamma=(\gamma_t)_{t\in T}$ in $\mathcal{Z}$.
Thus, 
for all $t\in T$, 
$\gamma_t\in C^\mathcal{F}(\varphi_t(K_t),\mathbb{R}^d)$ 
and 
\begin{equation}
\lim_{k\to\infty}\left\|\gamma_t-\psi_t\circ g_k\circ \varphi_t^{-1}\right\|_{\varphi_t(K_t)}
=
0 \ .
\end{equation}
First, let us show that $\gamma$ is realised by a continuous mapping, 
{\it i.e.}, 
there exists $g\in C^0(M,N)$ such that $\gamma=\iota(g)$.
Observe that 
$\gamma_t|_{\interior{K}_t}$ is continuous and 
$\{\interior{K}_t\}_{t\in T}$ cover $M$.
Therefore it suffices to show that, 
for any $s,t\in T$, 
\begin{equation}\label{eq:gammat=gammas}
\gamma_t|_{\varphi_t(\interior{K}_s\cap\interior{K}_t)}
=
\gamma_s|_{\varphi_s(\interior{K}_s\cap\interior{K}_t)}
\end{equation}
If we define
$\varphi_{t,s}=\varphi_t\circ\varphi_s^{-1}$
and
$\psi_{t,s}=\psi_t\circ\psi_s^{-1}$, whenever they are defined,
then it suffices to show that, for all $s,t\in T$,
\begin{equation}\label{eq:gammat=gammas_in_t-coord}
\gamma_t|_{\varphi_t(K_s\cap K_t)}
=
\psi_{t,s}\circ\gamma_s\circ\varphi_{t,s}^{-1}|_{\varphi_t(K_s\cap K_t)} \ .
\end{equation}
%(These are well-defined statements as property (iv) implies that 
%$\gamma_t\circ\varphi_t(z)\in V_t$, for each $t\in  T$.)
Fix a positive integer $k$.
By the triangle inequality, 
followed by the Monotonicity property (III) 
and the Composition property (I),
\begin{align}
&
\|\gamma_t-\psi_{t,s}\circ\gamma_s\circ\varphi_{t,s}^{-1}\|_{\varphi_t(K_s\cap K_t)}\notag\\
&\leq
\|\gamma_t-\psi_t\circ g_k\circ\varphi_t^{-1}\|_{\varphi_t(K_s\cap K_t)}\\
&\quad+
\|\psi_t \circ g_k\circ \varphi_t^{-1}-\psi_{t,s}\circ\gamma_s\circ\varphi_{t,s}^{-1}\|_{\varphi_t(K_s\cap K_t)}\notag\\
&\leq 
\|\gamma_t-\psi_t \circ g_k\circ\varphi_t^{-1}\|_{\varphi_t(K_t)}\\
&\quad+
\|\psi_{t,s}\circ\psi_s \circ g_k \circ\varphi_s^{-1}\circ\varphi_{t,s}^{-1}-\psi_{t,s}\circ\gamma_s\circ\varphi_{t,s}^{-1}\|_{\varphi_t(K_s\cap K_t)}\notag\\
&\leq
\|\gamma_t-\psi_t \circ g_k\circ \varphi_t^{-1}\|_{\varphi_t(K_t)}
+
C_{s,t}\|\psi_s \circ g_k \circ\varphi_s^{-1}-\gamma_s\|_{\varphi_s(K_s)}
\end{align}
where $C_{s,t}$ is a positive real number independent of $k$.
Since, for any $t\in T$, 
$\|\gamma_t-\psi_t \circ g_k\circ \varphi_t^{-1}\|_{\varphi_t(K_t)}$ 
can be made arbitrarily small by choosing $k$ sufficiently large, it follows that
$\|\gamma_t-\psi_{t,s}\circ\gamma_s\circ\varphi_{t,s}^{-1}\|_{\varphi_t(K_s\cap K_t)}=0$.
Consequently, equality~\eqref{eq:gammat=gammas_in_t-coord} holds. 
Since $s, t\in T$ were arbitrary,
$\gamma$ is realised by a mapping $g\in C^0(M,N)$, as required.

Next, let us show that $g\in C^\mathcal{F}(M,N)$.
By definition, this means showing that, for any charts 
$(U,\varphi)$ and $(V,\psi)$ of $M$ and $N$ respectively, 
and any compact subset $K$ of $U$,
the norm 
$\|\psi\circ g\circ \varphi^{-1}\|_{\varphi(K)}$ 
is finite.
Since $g\in C^0(M,N)$ and $K$ is compact in $U$,
\begin{equation}
\|\psi\circ g\circ\varphi^{-1}\|_{C^0(\varphi(K),\mathbb{R}^d)}<\infty \ .
\end{equation}
Thus, we only need to consider the $\mathcal{F}$-semi-norms.
Observe that 
the local finiteness of the collection of charts, 
together with the compactness of $K$,
implies that only finitely many $K_t$, $t\in T$, intersect $K$.
Thus, by the Gluing property, 
there exists a positive real number $C$ such that
\begin{equation}
\left[\psi\circ g\circ\varphi^{-1}\right]_{\varphi(K)}
\leq 
C\sum_{t\in T} 
\left[\psi\circ g\circ\varphi^{-1}\right]_{\varphi(K_t\cap K)} \ .
\end{equation}
(We remind the reader that this sum is finite.)
The Composition property (I) implies that, 
for some positive real number $\kappa_t$,
\begin{align}
\left[\psi\circ g\circ\varphi^{-1}\right]_{\varphi(K_t\cap K)}
&=
\left[(\psi\circ\psi_t^{-1})\circ\psi_t\circ g\circ\varphi_t^{-1}(\varphi_t\circ \varphi^{-1})\right]_{\varphi(K_t\cap K)}\\
&\leq
\kappa_t \left\|\psi_t\circ g\circ\varphi_t^{-1}\right\|_{\varphi_t(K_t\cap K)} \ .
\end{align}
Together with the Monotonicity property (III) this therefore implies that
\begin{align}
\left[\psi\circ g\circ\varphi^{-1}\right]_{\varphi(K_t\cap K)}
&\leq
\kappa_t \left\|\psi_t\circ g\circ\varphi_t^{-1}\right\|_{\varphi_t(K_t)} \ .
\end{align}
Convergence in $\mathcal{Z}$ implies that, for $k$ sufficiently large,
\begin{equation}
\left\|\psi_t\circ g\circ\varphi_t^{-1}\right\|_{\varphi_t(K_t)}
\leq
\left\|\psi_t\circ g_k\circ\varphi_t^{-1}\right\|_{\varphi_t(K_t)}
+1
<\infty \ .
\end{equation}
Consequently
$\left\|\psi_t\circ g\circ\varphi_t^{-1}\right\|_{\varphi_t(K_t)}$ 
is finite, for all $t\in T$.
Hence, the semi-norm 
$\left[\psi\circ g\circ\varphi^{-1}\right]_{\varphi(K)}$ 
is finite.
Thus
$g\in C^\mathcal{F}(M,N)$, and the claim is shown.
/\!/

\vspace{10pt}

\noindent
%%%%%%%%%%%%%
%\end{comment}
%%%%%%%%%%%%%
Since a closed subspace of a complete metric space 
is also a complete metric space, 
it follows that the pullback metric $\rho$ 
of the induced metric on 
$\iota(\mathcal{B})$ is complete.
%%%%%%%%%%%%%%%
%\end{comment}
%%%%%%%%%%%%%%%
\end{proof}
%
%%%%%%%%%%%%%%%%%%%%%%%%%%%%%
\subsection{The main result} \label{subsect:main}
%%%%%%%%%%%%%%%%%%%%%%%%%%%%%
%
In this section we state and prove our main result, namely Theorem \ref{thm:Whitney_topologies_are_Baire}. 
Before we do that, however, we still need the following auxiliary result. 

\begin{proposition}\label{prop:F-Whitney_sprinkled}
Let 
$\mathcal{B}$ be a special sub-basic neighbourhood and 
$\rho$ be the complete metric for $\mathcal{B}$ as given in section~\ref{sect:subbasic_sets} 
%Lemma~\ref{lem:complete_metric} 
above.
\begin{enumerate}
\item\label{prop:*+weak-Whitney}
Assume $T$ is finite.
Then for
each open set $\mathcal{B}_1$ in the weak $\mathcal{F}$-Whitney topology 
which intersects $\mathcal{B}$, 
the intersection $\mathcal{B}\cap\mathcal{B}_1$ 
contains the $\rho$-closure of an open ball with respect to the metric $\rho$, 
and 
this open ball is also a basic set in the weak $\mathcal{F}$-Whitney topology.
Moreover, this ball can be chosen with arbitrarily small diameter with respect to $\rho$.
\item\label{prop:**+strong-Whitney}
Assume $T$ is countable.
Then for
each open set $\mathcal{B}_1$ in the strong $\mathcal{F}$-Whitney topology, 
which intersects $\mathcal{B}$, 
the intersection $\mathcal{B}\cap\mathcal{B}_1$
contains the $\rho$-closure of a basic set in the strong $\mathcal{F}$-Whitney topology 
with arbitrarily small diameter with respect to $\rho$.
\end{enumerate}
\end{proposition}
\begin{proof}
%%%%%%%%%%%%%%%
%\begin{comment}
%%%%%%%%%%%%%%%
%%%%%%%%%%%%%%%%%%%WEAK%%%%%%%%%%%%%%%%%%%%%%%%%
Consider the first statement on the weak $\mathcal{F}$-Whitney topology.
Take an arbitrary non-empty open set $\mathcal{B}_1$. 
We may assume that $\mathcal{B}_1$ is contained in $\mathcal{B}$ and also that $\mathcal{B}_1$ is a weak basic set, 
{\it i.e.},
\begin{equation}
\mathcal{B}_1
=
\bigcap_{s\in S}
\mathcal{N}(g_s;(U_s',\varphi_s'),(V_s',\psi_s'),K_s',\epsilon_s')
\end{equation}
for some finite index set $S$.
Take $g_1\in\mathcal{B}_1$.
For $s\in S$, define
\begin{equation}
r_s=
\epsilon_s'-\|\psi_s'\circ g_1\circ(\varphi_s')^{-1}-\psi_s'\circ g_s\circ(\varphi_s')^{-1}\|_{\varphi_s'(K_s')} \ .
\end{equation}
Since $g_1\in \mathcal{B}_1$, the number $r_s$ is positive.
We will show that, for each $s\in S$, there exists a positive real number $\varrho_s$ for which
\begin{equation}\label{incl:Neps_s}
\bigcap_{t\in T}
\mathcal{N}(g_1;(U_t,\varphi_t),(V_t,\psi_t),K_t,\varrho_s)
\subset
\mathcal{N}(g_s;(U_s',\varphi_s'),(V_s',\psi_s'),K_s',\epsilon_s') \ .
\end{equation} 
By Proposition~\ref{prop:change_of_charts-local}, as the index set $T$ is finite,
it follows that there exists $\varrho_s$ sufficiently small for which $g$ satisfying
\begin{align}
\|\psi_t\circ g_1\circ\varphi_t^{-1}-\psi_t\circ g\circ\varphi_t^{-1}\|_{\varphi_t(K_t)}<\varrho_s
\end{align}
for all $t\in T$, 
implies that
\begin{equation}
\|\psi_s'\circ g_1\circ(\varphi_s')^{-1}-\psi_s'\circ g\circ(\varphi_s')^{-1}\|_{\varphi_s'(K_s')}
<
r_s \ .
\end{equation}
Consequently
\begin{align}
&
\|\psi_s'\circ g_s\circ(\varphi_s')^{-1}-\psi_s'\circ g\circ(\varphi_s')^{-1}\|_{\varphi_s'(K_s')}
%\\
%&\leq 
%\|\psi_s'\circ g_s\circ(\varphi_s')^{-1}-\psi_s'\circ g_1\circ(\varphi_s')^{-1}\|_{\varphi_s'(K_s')}
%+
%\|\psi_s'\circ g_1\circ(\varphi_s')^{-1}-\psi_s'\circ g\circ(\varphi_s')^{-1}\|_{\varphi_s'(K_s')}
<\epsilon_s \ ,
\end{align}
and thus the inclusion~\eqref{incl:Neps_s} is shown for each $s\in S$.
Observe that the lefthand-side of~\eqref{incl:Neps_s} is just the $\varrho_s$-ball about $g_1$, with respect to the metric $\rho$.
Since the finite intersection of concentric non-empty balls is a ball, it follows that $\mathcal{B}_1$ contains a ball about $g_1$.
Replacing $\varrho_s$ in the above argument by any small positive real number also produces a ball contained in $\mathcal{B}_1$ about $g_1$ with 
small diameter, which shows the last statement.

%%%%%%%%%%%%%%%%%%%%%STRONG%%%%%%%%%%%%%%%%%%%%%%%%%%%%%%%%
Next, consider the strong $\mathcal{F}$-Whitney topology.
Recall that $\mathcal{B}$ has the form
\begin{equation}
\mathcal{B}=\bigcap_{t\in T}\overline{\mathcal{N}}(f,(U_t,\varphi_t),(V_t,\psi_t),K_t,\epsilon_t)
\end{equation}
for a locally finite collection of charts $\Phi=\{(U_t,\varphi_t)\}_{t\in T}$ on $M$, etc., 
and thus contains 
$\mathcal{N}_s(f;\Phi,\Psi,K,\epsilon)$ 
densely.
We may assume that $\mathcal{B}_1$ is a strong basic neighbourhood, 
{\it i.e.},
\begin{equation}
\mathcal{B}_1=\mathcal{N}_s(g';\Phi',\Psi',K',\epsilon')
\end{equation}
where 
$\Phi'=\{(U_s',\varphi_s')\}_{s\in S}$ 
is a locally finite collection of charts on $M$, 
$\Psi'=\{(V_s',\psi_s')\}_{s\in S}$ 
is a collection of charts on $N$, 
$K'=\{K_s'\}_{s\in S}$ 
is a collection of compact sets with $K_s'\subset U_s'$ and $g(K_s')\subset V_s'$. 
and 
$\epsilon'=\{\epsilon_s'\}_{s\in S}$
is a collection of extended positive real numbers.
As the open set 
$\mathcal{N}_s(f;\Phi,\Psi,K,\epsilon)$ 
is dense in $\mathcal{B}$ and $\mathcal{B}_1$ is open,
if 
$\mathcal{B}_1$ intersects $\mathcal{B}$ then so does 
$\mathcal{N}_s(f;\Phi,\Psi,K,\epsilon)$.
Take 
$g\in \mathcal{B}_1\cap \mathcal{N}_s(f;\Phi,\Psi,K,\epsilon)$.
By Lemma~\ref{lem:change_of_charts}, 
there exists a family of extended positive real numbers
$\delta=\{\delta_t\}_{t\in T}$ 
such that for
\begin{equation}
\mathcal{B}_2=\mathcal{N}_s(g;\Phi,\Psi,K,\delta)
\end{equation}
we have 
$\mathcal{B}_2\subset \mathcal{B}_1\cap \mathcal{N}_s(f;\Phi,\Psi,K,\epsilon)$.
Also observe that shrinking each $\delta_t$ does not affect this inclusion.
Finally, observing that 
$\mathrm{diam}_\rho(\mathcal{B}_2)=\sup_{t\in T} 2\delta_t$,
gives the result.
%%%%%%%%%%%%%
%\end{comment}
%%%%%%%%%%%%%
\end{proof}
We are now in a position to prove the following.
\begin{theorem}\label{thm:Whitney_topologies_are_Baire}
Let $M$ and $N$ be smooth manifolds.
Let $\mathcal{F}$ be a family of semi-norms satisfying properties (I)--(III).
Then both the weak and strong $\mathcal{F}$-Whitney topologies on $C^\mathcal{F}(M,N)$ 
satisfy the Baire property.
\end{theorem}
\begin{proof}
We will only consider the weak $\mathcal{F}$-Whitney topology.
(The proof for the strong $\mathcal{F}$-Whitney topology follows {\it mutatis mutandis}.)
By Lemma~\ref{lem:opens_contain_specials}, 
any open set $\mathcal{U}$ in the weak $\mathcal{F}$-Whitney topology 
contains a special sub-basic neighbourhood $\mathcal{B}$.
By Proposition~\ref{prop:barN-closed}, each special sub-basic neighbourhood 
has non-empty interior $\mathcal{B}_0$ with respect to the weak $\mathcal{F}$-Whitney topology.
By Lemma~\ref{lem:complete_metric}, 
any special sub-basic neighbourhood $\mathcal{B}$ 
possesses a complete metric $\rho$. 
Proposition~\ref{prop:F-Whitney_sprinkled}(1) implies that the interior $\mathcal{B}_0$ of $\mathcal{B}$ is $\rho$-sprinkled.
Since $\mathcal{U}$ was arbitrary, the weak $\mathcal{F}$-Whitney topology is locally sprinkled and hence,
by Lemma~\ref{lem:locally_sprinkled}, it satisfies the Baire property.
\end{proof}

%%%%%%%%%%%%%%%%%%
\section{Examples}\label{sect:examples}
%%%%%%%%%%%%%%%%%%
In the examples below, $M$ and $N$ are 
smooth manifolds of dimensions $m$ and $d$ respectively.
Given normed linear spaces $V$ and $W$ of dimension $m$ and $d$ respectively, 
let 
$L^n_\mathrm{sym}(V,W)$ 
denote space of symmetric $n$-linear maps from $V$ to $W$.
For each $n$, 
abusing notation slightly,
we denote the operator norm on 
$L^n_\mathrm{sym}(V,W)$ 
by $|\, \cdot \,|$.
We could also take any equivalent norm with the property that $|AB|\leq |A|\cdot|B|$. 
%For instance, for $n=1$ $|\,\cdot\,|$ denotes the matrix norm 
%$|A|=\sum_{i=1,\ldots,d; j=1,\ldots,m} |a_{i,j}|$
%(Due to finite dimensionality, all such norms are equivalent.)
%%%%%%%%%%%%%%%%%%%%%%%%%%%%%%%%%%%%%%%%%%%%%%%%%%%%%%%
\subsection{H\"older- and Lipschitz-Whitney topologies}
%%%%%%%%%%%%%%%%%%%%%%%%%%%%%%%%%%%%%%%%%%%%%%%%%%%%%%%
For each integer $k\geq 0$, 
consider 
the family of {\it $C^k$-semi-norms}
\begin{equation}
[f]_{C^k,K}=\max_{0<j\leq k}\sup_{x\in  K}|D^j f(x)|
\end{equation}
where $K\subset\mathbb{R}^m$ is compact and $f$ is a $C^k$-mapping from a neighbourhood of $K$ to $\mathbb{R}^d$.
The weak and strong generalized Whitney topologies for this family of semi-norms
coincide respectively with the classical weak and strong $C^k$-Whitney topologies on $C^k(M,N)$. 
\begin{proposition}\label{prop:Ck-Whitney_is_I--III}
The family of $C^k$-semi-norms satisfies properties (I)--(III).
\end{proposition}
\begin{proof}
Throughout we take $x\in K$ and $y=\varphi(x)$.
First, to prove (I) we break it into two parts:
\begin{enumerate}
\item[(i)]
$[g\circ \varphi^{-1}-g'\circ \varphi^{-1}]_{C^k,K}\leq \kappa_1 [g-g']_{C^k,K} \, $;
\item[(ii)] 
$[\psi\circ g-\psi\circ g']_{C^k,K}\leq \kappa_2 \left(\|g-g'\|_{C^0(K,\mathbb{R}^d)}+[g-g']_{C^k,K}\right) \, $.
\end{enumerate}
for constants $\kappa_1$ and $\kappa_2$.
The Composition Mapping Formula~\cite[p.3]{AbrahamRobbinBook1967} %Theorem 1.4
implies that
\begin{align}
|D^k(g\circ \varphi^{-1})(y)-D^k(g\circ\varphi^{-1})(y)|
%&\leq
%\sum_{1\leq j\leq k}
%\sum_{i=(i_1,\ldots,i_j); |i|=k}
%\sigma_k(i)|D^j g (x) (D^{i_1}(\varphi^{-1})(y),\ldots,D^{i_j}(\varphi^{-1})(y)) - D^j g' (x) (D^{i_1}(\varphi^{-1})(y),\ldots,D^{i_j}(\varphi^{-1})(y))|\\
&\leq 
k K_1 \sup_{0<j\leq k}|D^j g(x)-D^jg'(x)|
\end{align}
where $K_1$ depends only upon $k$. 
Taking the supremum over all $x\in K$ gives (i).
For (ii), by the Composition Mapping Formula again, it suffices to show that, for any multi-index $i=(i_1,\ldots,i_j)$ with $|i|=k$,
\begin{equation}
\left|
D^j\psi(g(x))\left(D^{i_1}g(x),\ldots,D^{i_j}g(x)\right)-D^j\psi(g'(x))\left(D^{i_1}g'(x),\ldots,D^{i_j}g'(x)\right)
\right|
\end{equation}
is bounded from above by $\kappa_2 [g-g']_{C^k,K}$, for some $\kappa_2$ independent of $x\in K$.
Writing this as a telescoped sum and applying the triangle inequality gives the upper bound
\begin{align}
&|D^j\psi(g(x))-D^j\psi(g'(x))|\notag\\
&\quad+\sum_{1\leq \ell\leq j}\left|D^j\psi(g'(x))\left(D^{i_1}g'(x),\ldots,D^{i_\ell}(g-g')(x),\ldots,D^{i_j}g(x)\right)\right|\\
&\leq
|D^{j+1}\psi|_{\bar{U}}|g(x)-g'(x)|\notag\\
&\quad+\sum_{1\leq \ell\leq j} |D^j\psi(g'(x))| |D^{i_1}g'(x)|\ldots |D^{i_\ell}(g-g')(x)|\ldots |D^{i_j}g(x)|\\
&\leq
K_2(\|g-g'\|_{C^0(K,\mathbb{R}^d)}+[g-g']_{C^k,K}) \ .
\end{align}
The second inequality follows from the Mean Value Theorem~\cite[p.4]{AbrahamRobbinBook1967}%Theorem 1.6
, where $U$ is a bounded open set containing the convex hull of $g(K)\cup g'(K)$.
Here $K_2$ depends only upon $k$, $[\psi]_{C^{k+1},K}$ and $\max\left\{[g]_{C^k,K}, [g']_{C^k,K}\right\}$, but not on the point $x\in K$.

Finally, Properties (II) and (III) hold as the sup-norm satisfies these properties.
\end{proof}
For $\alpha\in (0,1)$, 
consider the family of
{\it $\alpha$-H\"older semi-norms} defined by
\begin{equation}
\left[g\right]_{\alpha,K}
=
\sup_{x,y\in K: x\neq y}\frac{\left|g(x)-g(y)\right|_E}{d(x,y)^\alpha}
\end{equation}
where, again, $K\subset\mathbb{R}^m$ is compact and $g$ is an $\alpha$-H\"older mapping from $K$ into a Banach space $(E,|\,\cdot\,|_E)$.
In practice, this Banach space will be $L^k_\mathrm{sym}(\mathbb{R}^m,\mathbb{R}^d)$.
Define the family of 
{\it Lipschitz semi-norms} $[\,\cdot\,]_{\mathrm{Lip},K}$ similarly. 
\footnote{We use this notation, rather than $[\,\cdot\,]_{1,K}$, to prevent possible ambiguity.}
\begin{proposition}\label{prop:Holder-Whitney_is_I--III}
For each $\alpha\in (0,1)$, the family of $\alpha$-H\"older and Lipschitz semi-norms satisfies properties (I)--(III). 
\end{proposition}
\begin{proof}
In both cases, Property (I) follows from the H\"older Rescaling Principle~\cite[Proposition 2.2]{deFariaHazardTresser2017b}.
Property (II) is a straightforward corollary of the Second H\"older Gluing Principle~\cite[Proposition 2.4]{deFariaHazardTresser2017b} in the H\"older case, 
and the Lipschitz Gluing Principle~\cite[Lemma B.1]{deFariaHazardTresser2017b} in the Lipschitz case.
Property (III) follows trivially from the definition of the $\alpha$-H\"older and Lipschitz semi-norms.
\end{proof}
For each integer $k\geq 0$ and $\alpha\in (0,1)$, now consider the family of {\it $C^{k+\alpha}$-semi-norms} defined by 
\begin{equation}
[f]_{C^{k+\alpha},K}=[f]_{C^k,K}+[D^k f]_{\alpha,K}
\end{equation}
where, as before, 
$K\subset\mathbb{R}^m$ is compact and 
$f$ is a $C^k$-mapping from a neighbourhood of $K$ to $\mathbb{R}^d$, 
whose $k$th derivative is $\alpha$-H\"older continuous on $K$.
Define the family of {\it $C^{k+\mathrm{Lip}}$-semi-norms} 
$[\,\cdot\,]_{C^{k+\mathrm{Lip}},K}$ similarly.

Given manifolds $M$ and $N$ above, the weak and strong generalized Whitney topologies 
for the family of $C^{k+\alpha}$-semi-norms and will be called respectively
the {\it weak and strong $C^{k+\alpha}$-Whitney topologies}. 
The {\it weak and strong $C^{k+\mathrm{Lip}}$-Whitney topologies} are defined analogously.
In~\cite{deFariaHazardTresser2017b}, 
for $k=0$, these were termed the 
{\it $\alpha$-H\"older-Whitney} and the {\it Lipschitz-Whitney topologies} respectively. 

By Remark~\ref{rmk:sums_of_semi-norm_families}, 
Propositions~\ref{prop:Ck-Whitney_is_I--III} and~\ref{prop:Holder-Whitney_is_I--III}
imply that
the family of $C^{k+\alpha}$-semi-norms and $C^{k+\mathrm{Lip}}$-semi-norms also satisfy properties (I)--(III).
By Theorem~\ref{thm:Whitney_topologies_are_Baire} this implies the following.
\begin{theorem}
For each integer $k\geq 0$ and $\alpha\in [0,1)$, the weak and strong $C^{k+\alpha}$-Whitney topologies 
%on $C^\alpha(M,N)$ 
are Baire. 
The weak and strong $C^{k+\mathrm{Lip}}$-Whitney topologies 
%on $C^\mathrm{Lip}(M,N)$ 
are also Baire.
\end{theorem}
\begin{remark}
For $k=0$
this result, in a slightly modified form, was stated as~\cite[Proposition 2.1]{deFariaHazardTresser2017b}. 
(There it is stated for bi-$\alpha$-H\"older homeomorphisms but, as mentioned in Remark~\ref{rmk:mappings_not_homeos}, 
the same argument goes through in both cases with obvious changes.)
\end{remark}
%
\begin{comment}
The more well-known argument, given in~\cite{HirschBook}, uses a construction involving jet spaces.
However, it is unclear to us if this argument may be modified for the case $k=0$.
\end{comment}
%
%%%%%%%%%%%%%%%%%%%%%%%%%%%%%%%%%%%%%%%
\subsection{Sobolev-Whitney topologies}
%%%%%%%%%%%%%%%%%%%%%%%%%%%%%%%%%%%%%%%
We now consider the generalized Whitney topology for families of Sobolev norms.
For background on Sobolev mappings in Euclidean spaces we suggest~\cite{Maz'yaBook,ZiemerBook}.

For each integer $k\geq 1$ and $1\leq p< \infty$, consider the family of {\it $W^{k,p}$-semi-norms}
\begin{equation}
[f]_{k,p,K}
=
\left(\int_K \sum_{0<j\leq k}|D^jf(x)|^p d\mu(x)\right)^{\frac{1}{p}}
\end{equation}
where 
$K\subset \mathbb{R}^m$ is compact,
$\mu$ denotes Lebesgue measure on $\mathbb{R}^m$, and 
$f$ is a continuous function from $K$ to $\mathbb{R}^d$, whose components $f_1,f_2,\ldots,f_d$ all lie in the Sobolev class $W^{k,p}$.
Recall that this means that 
all weak derivatives of order $k$ or less exist in some neighbourhood of $K$, and 
these weak derivatives lie in $L^p_\mu(K)$.
Denote by 
$W^{k,p}(K,\mathbb{R}^d)$ the space of all $W^{k,p}$-Sobolev maps between $K$ and $\mathbb{R}^d$,
and let 
$\mathbb{W}^{k,p}(K,\mathbb{R}^d)=W^{k,p}(K,\mathbb{R}^d)\cap C^0(K,\mathbb{R}^d)$.
Endowed with the norm
\begin{equation}
\|\cdot\|_{\mathbb{W}^{k,p}(K,\mathbb{R}^d)}
=
\|\cdot\|_{C^0(K,\mathbb{R}^d)}
+
[\,\cdot\,]_{k,p,K} \ ,
\end{equation}
the space 
$\mathbb{W}^{k,p}(K,\mathbb{R}^d)$ is a Banach space.

\begin{proposition}\label{prop:Sobolev-Whitney_is_I--III}
For $p>m$, the family of $W^{k,p}$-semi-norms satisfies properties (I)--(III).
\end{proposition}
\begin{proof}
First consider (I) in the case $k=1$.
By~\cite{Cesari1941,Calderon1951}, 
maps of Sobolev class $W^{1,p}$ are differentiable 
Lebesgue almost everywhere, for $p>m$.
By the same argument as~\cite[Lemma 3.1]{deFariaHazardTresser2017a} 
this implies a Sobolev Chain Rule for pre- or post-composition by $C^1$-diffeomorphisms.
Now we break the proof into two parts:
\begin{enumerate}
\item[(i)]
$[g\circ \varphi-g'\circ\varphi]_{C^k,\varphi^{-1}(K)}\leq \kappa_1 [g-g']_{C^k,K}$
\item[(ii)]
$[\psi\circ g-\psi\circ g']_{C^k,\varphi(K)}\leq \kappa_2 \left(\|g-g'\|_{C^0(K,\mathbb{R}^d)}+[g-g']_{C^k,K}\right)$
\end{enumerate}
For (i), the Sobolev Chain Rule, together with the Change of Variable Formula for the Lebesgue integral gives
\begin{align}
&
\left[g\circ \varphi-g'\circ\varphi\right]_{1,p,\varphi^{-1}(K)}\notag\\
&=
\left(\int_{\varphi^{-1}(K)} 
|Dg(\varphi(x)) D\varphi(x) - Dg'(\varphi(x)) D\varphi(x)|^p
d\mu(x)\right)^{\frac{1}{p}}\\
&\leq  
\frac{\max\|D\varphi(x)\|}{\min |\mathrm{Jac} \; \varphi(x)|^{\frac{1}{p}}}
\cdot
\left(\int_{K} |Dg(y)- Dg'(y)|^pd\mu(y)\right)^{\frac{1}{p}} 
\end{align}
and (i) follows as $\varphi$ is a diffeomorphism and $\varphi^{-1}(K)$ is compact.
For (ii), by the Sobolev Chain Rule, then
telescoping the sum and applying triangle inequality, 
\begin{align}
[\psi\circ g-\psi\circ g']_{1,p,K}
&=
\left(\int_K 
|D\psi(g(x)) Dg(x) - D\psi(g'(x)) Dg'(x)|^p
d\mu\right)^{\frac{1}{p}}\\
&\leq 
\left(\int_K 
|D\psi(g(x)) - D\psi(g'(x))|^p |Dg(x)|^p
d\mu\right)^{\frac{1}{p}}\notag\\
&\qquad+\left(\int_K 
|D\psi(g(x))|^p |Dg(x) - Dg'(x)|^p
d\mu\right)^{\frac{1}{p}} \ .
\end{align}
By H\"older's inequality this is bounded from above by
\begin{align}
&
\max_{x\in K} |D\psi\circ g(x)-D\psi\circ g'(x)| \cdot [g]_{1,p,K}
+\max_{y\in g(K)} |D\psi(y)| \cdot [g-g']_{1,p,K}\notag\\
&\leq
\max_{x\in\bar{U}}|D^2\psi(x)| \cdot [g]_{1,p,K}\cdot \|g-g'\|_{C^0(K,\mathbb{R}^d)}
+\max_{y\in g(K)}|D\psi(y)| \cdot [g-g']_{1,p,K} \ .
\end{align}
The last inequality follows from the Mean Value Theorem~\cite[p.4]{AbrahamRobbinBook1967}, 
where $U$ is an open neighbourhood of the convex hull of $g(K)\cup g'(K)$.
Property (I) now follows.

Next, let us sketch the proof for arbitrary $k$. 
Observe that if $p>m$, 
then any $j$th order derivative is of class $W^{k-j,p}$, 
and thus also differentiable Lebesgue almost everywhere.  
The argument used for the Sobolev Chain Rule, 
together with the Composition Mapping Formula~\cite[p.3]{AbrahamRobbinBook1967} 
gives a Sobolev Composition Mapping Formula. 
Using this together with a natural modification 
to the argument above (and also Proposition~\ref{prop:Ck-Whitney_is_I--III}) shows that (I) holds.

Finally, (II) and (III) follow trivially from the definition of the $W^{k,p}$-semi-norm.
%Recall that the number of compact sets $K_s$ when gluing is finite and fixed.
\end{proof}
Given manifolds $M$ and $N$ above, 
the weak and strong generalized Whitney topologies, 
corresponding to the family of $W^{k,p}$-semi-norms, are called respectively 
the {\it weak and strong $W^{k,p}$-Whitney topology} 
or {\it $(k,p)$-Sobolev-Whitney topology}.
The following Theorem is now a direct consequence of Proposition~\ref{prop:Sobolev-Whitney_is_I--III}.
\begin{theorem}
For $p>m$, the weak and strong $(k,p)$-Sobolev-Whitney topologies 
%on $C^{W^{1,p}}(M,N)$ 
are Baire. 
\end{theorem}
\begin{remark}
When considering open mappings or homeomorphisms, rather than arbitrary mappings,
either 
(a) $m>2$ and $p>m-1$, or
(b) $m=2$ and $p\geq 1$,
is sufficient to show that $(k,p)$-Sobolev semi-norms satisfy properties (I)--(III).
The main point being that these conditions ensure, by~\cite{GehringLehto1959,Vaisala1965},
that maps of Sobolev class $W^{k,p}$ are differentiable Lebesgue almost everywhere.
This is enough to ensure~\cite[Lemma 3.1]{deFariaHazardTresser2017a} that the Sobolev Chain Rule holds.  

%As noted in Remark~\ref{rmk:mappings_not_homeos},
Using this, it then follows that the Baire property is satisfied in the space of
open $(k,p)$-Sobolev maps or homeomorphisms, 
with $(k^*,p^*)$-Sobolev inverse, under the weaker conditions
of either $p$ and $p^*$ satisfying either (a) or (b).
(This was previously stated in~\cite[Proposition 3.1]{deFariaHazardTresser2017b} in the case $k=k^*=1$.)
\end{remark}
\begin{remark}
The more familiar topology on spaces of Sobolev mappings~\cite{SchoenUhlenbeck1984,Bethuel1991} 
requires the manifolds to be endowed with a Riemannian structure, 
and then uses the isometric embeddings into Euclidean spaces.
The Sobolev-Whitney topologies described above do not depend upon a Riemannian structure. 
%We do not know if there is a relation between these two topologies.
\end{remark}

%%%%%%%%%%%%%%%%%%%%%%%%%%%%%%%%%%%%%%%%%%%%%%

\section*{Acknowledgements}
We would like to thank IME-USP for their constant support.
Thanks also go to Mario Bessa, Charles Tresser and Michael Benedicks for 
their questions and continued interest in this work.

\end{document}